\documentclass[12pt,reqno]{amsart}
\usepackage{amssymb}
\usepackage{mathrsfs}
\usepackage{graphicx}
\usepackage[headings]{fullpage}

\numberwithin{equation}{section}

\theoremstyle{plain}

\newtheorem{theorem}[subsubsection]{{Theorem}}

\newtheorem{proposition}[subsubsection]{{Proposition}}
\newtheorem{corollary}[subsubsection]{{Corollary}}

\theoremstyle{definition}
\newtheorem{definition}[subsubsection]{{Definition}}
\newtheorem{example}[subsubsection]{{Example}}

\theoremstyle{remark}
\newtheorem{remark}[subsubsection]{{Remark}}

\newcommand{\Z}{\mathbb{Z}}
\newcommand{\F}{\mathbb{F}}
\newcommand{\Q}{\mathbb{Q}}
\newcommand{\R}{\mathbb{R}}

\newcommand{\Aut}{\textnormal{Aut}}
\newcommand{\Char}{\textnormal{char}}
\newcommand{\End}{\textnormal{End}}
\newcommand{\Frac}{\textnormal{Frac}}
\newcommand{\Gal}{\textnormal{Gal}}
\newcommand{\GL}{\textnormal{GL}}
\newcommand{\Hom}{\textnormal{Hom}}
\newcommand{\Ht}{\textnormal{ht}}
\newcommand{\Mat}{\textnormal{Mat}}
\newcommand{\Nm}{\textnormal{Nm}}
\newcommand{\Spec}{\textnormal{Spec}}
\newcommand{\sep}{\textnormal{sep}}

\begin{document}

\title{Local Monodromy of $1$-Dimensional $p$-Divisible Groups}

\author{Tristan Phillips}
\address{Department of Mathematics \\ Dartmouth College
  \\ Hanover, NH~~03755 USA}
\email{tristanphillips72@gmail.com}


\begin{abstract}
Let $G$ be a $p$-divisible  group over a complete discrete valuation ring $R$ of characteristic $p$. The generic fiber of $G$ determines a Galois representation $\rho$. The image of $\rho$ admits a ramification filtration and a Lie filtration. We relate these filtrations in the case $G$ is one dimensional, giving an equicharacteristic version of Sen's theorem in this setting. This result generalizes a result of Gross. Additionally, we prove that the representation associated to the \'etale part of $G$ is irreducible, generalizing a result of Chai.
\end{abstract}

\maketitle

\section{Introduction}
\label{intro}

Let $R$ be a complete discrete valuation ring whose fraction field $K$ has algebraically closed residue field $k$ of characteristic $p$. Let $L$ be a Galois extension of $K$. Then $\Gal(L/K)$ has a \emph{ramification filtration} which reflects certain arithmetic properties of $L/K$. If $\Gal(L/K)$ is also a $p$-adic Lie group, then $\Gal(L/K)$ also has a \emph{Lie filtration} which reflects certain analytic properties of $L/K$. One can ask how these two filtrations compare. As the Lie filtration is often more easily understood, a relation between these two filtrations gives an avenue to study certain arithmetic properties of $L/K$. 

In the case that $K$ has characteristic $0$, Serre conjectured a strong connection between the Lie and ramification filtrations \cite{Ser67}. This conjecture was proved by Wyman in the $\Z_p$ case \cite{Wym69}, and by Sen in the general case \cite{Sen72}. 

In the case that $K$ has characteristic $p>0$, the ramification filtration can exhibit very general behavior which prohibits any meaningful relation with the Lie filtration in general (see \cite[Remark 3]{GK88}). However, if one restricts to situations `coming from geometry' then some relations between the Lie and ramification filtrations in characteristic $p$ have been discovered. This was first explored by Gross for local monodromy groups of one-dimensional $p$-divisible groups \cite{Gro79}, and more recently by Kramer-Miller for monodromy groups of unit root $F$-isocrystals \cite{K-M18}. In this article we further explore the case of local monodromy groups of $p$-divisible groups, following in the footsteps of \cite{Gro79} and \cite{Cha00}.

 Let $G$ be a connected $1$-dimensional $p$-divisible group over $R$. It is known that the generic Newton polygon of $G$ has at most two slopes, $0$ and $1/g$, where $g\in \Z_{\geq 1}$ is a positive integer. The slope $0$ occurs with some multiplicity $d\in \Z_{\geq 0}$, and the slope $1/g$ occurs with multiplicity $1$. Then the height of $G$ is $s=g+d$. The action of the absolute Galois group $\mathfrak{G}_K:=\Gal(K^{\sep}/K)$ on the (generalized) Tate module associated to the generic fiber of $G$ determines a Galois representation
\[
\rho=\prod_\lambda \rho^\lambda: \mathfrak{G}_K\to \prod_{\lambda} \GL_{d_\lambda}(D_\lambda)
\]
where $D_\lambda$ is the $\Q_p$-division algebra with invariant $\lambda$. This representation will be described in more detail in Section \ref{sec:p-divisible groups}. The image of $\rho$, which we shall refer to as the \emph{local monodromy group of $G$}, is a closed subgroup of $\prod \GL_{d_\lambda}(D_\lambda)$ and thus inherits the structure of a $p$-adic Lie group. 

Suppose $G$ is $1$-dimensional. In this case $\rho=\rho^{0/1}\times \rho^{1/g}$. When $g=s-1$, Gross exhibited a relation between the ramification and Lie filtrations \cite{Gro79}. The case when $g=1$ and $s$ arbitrary was studied by Chai in \cite{Cha00}. However, there is a mistake in Chai's result comparing the filtrations \cite[Theorem 4.6]{Cha00} (see Appendix \ref{appendix} for a discussion of this mistake). In both of these cases an important role is played by an \emph{open image theorem}. In this setting the classical open image theorem is due to Igusa, whose results imply that when $s=2$ the image of $\rho^{0/1}$ is open in $\Z_p^\times$ \cite{Igu68}. Gross showed that when $g=s-1$, $\rho^{0/1}$ and $\rho^{1/g}$ are both surjective and thus have open image. Chai showed that when $g=1$ the image of $\rho^{1/1}=\det(\rho^{0/1})$ is open in $\Z_p^\times$. Beyond these cases, little is known about when $\rho$ has open image, though one can find some other related results at the end of \cite{Cha00} and in \cite{AN10}. We remark that there has, however, been much progress on the surjectivity of the local monodromy representations of \emph{universal} $p$-divisible groups \cite{Str10,Lau10,Tia09}.

The main result (Theorem \ref{thm:1/gBreaks}) of this article extends the result of Gross to all $g$ with $1\leq g<d$, under an open image assumption. In particular, we provide a fixed version of \cite[Theorem 4.6]{Cha00}. Our results give further support for an equicharacteristic version of Sen's theorem in this setting. We also prove that for any $1$-dimensional $p$-divisible group, the slope zero representation $\rho^{0/1}$ is irreducible, generalizing a result of Chai \cite{Cha00}. Our methods exploit the equivalence of categories between connected $p$-divisible groups and formal groups, which allows us to work very explicitly with formal power series. We then use Artin--Schreier theory to compute the ramification filtration.

\subsection{Organization}

Sections \ref{sec:p-adicLieGroups}-\ref{sec:formal groups} cover background material and our basic setup. In Section \ref{sec:p-adicLieGroups} we recall basic facts about $p$-adic Lie groups and their Lie filtrations. In Section \ref{sec:Ramification} we review higher ramification theory, covering basic facts about the lower and upper ramification filtrations. We also recall basic facts about Newton polygon and discuss ramification breaks in generalized Artin--Schreier extensions. In Section \ref{sec:p-divisible groups} we consider Galois representations of $p$-divisible groups. We review Gross's generalized Tate modules and the Galois representations arising from them. We also state an important theorem of Gross concerning Galois characters associated to those representations (Theorem \ref{thm:G2.7}). In Section \ref{sec:formal groups} we review the theory of formal groups and formal modules. We also show how to rephrase the Galois representations of $p$-divisible groups from Section \ref{sec:p-divisible groups} in terms of formal modules.

Sections \ref{sec:0/1-towers}-\ref{sec:GalChar} cover results on the local monodromy group of 1-dimensional p-divisible groups. In Section \ref{sec:0/1-towers} we study the $p$-adic Lie towers of local fields arising from $\rho^{0/1}$. Studying properties of these towers we are able to prove that the representation attached in $\rho^{0/1}$ is irreducible (Theorem \ref{thm:irreducible}). In Section \ref{sec:1/gTowers} we study $p$-adic Lie towers of local fields arising from $\rho^{1/g}$. We prove our main result, an equicharacteristic version of Sen's theorem for general 1-dimensional $p$-divisible groups (Theorem \ref{thm:1/gBreaks}). In Section \ref{sec:GalChar} we use Theorem \ref{thm:1/gBreaks} and Theorem \ref{thm:G2.7} to prove results on the ramification behavior of the Galois characters associated  to $\rho^{0/1}$ and $\rho^{1/g}$.

\section*{Acknowledgements}
This article came out of the authors comprehensive exam at the University of Arizona. Many thanks to Bryden Cais for suggesting this project and for giving much guidance and encouragement. The author would also like to thank Benedict Gross for clarifying a remark in \cite{Gro79}, and Anthony Kling for proofreading an earlier draft of this paper. The author is especially grateful for several referees for many helpful corrections and suggestions, and for pointing out the mistake in \cite{Cha00} and sketching how to fix it using Artin--Schreier theory. 
During the completion of this work the author was supported by the National Science Foundation, via grant DMS-2303011.

\section{$p$-Adic Lie Groups}\label{sec:p-adicLieGroups}

We cover some basic facts about $p$-adic Lie groups and their Lie filtrations. We refer the reader to \cite{DdMS99} for details.

\begin{definition}[Uniform pro-$p$ group]
A pro-$p$ group $H$ is \emph{uniform} if it satisfies the following conditions:
\begin{itemize}
\item[(i)] $H$ is finitely generated.
\item[(ii)] $H^p$, the subgroup generated by $\{h^p:h\in H\}$, is a normal subgroup of $H$ and $H/H^p$ is abelian. 
\item[(iii)] For all $i\geq 0$ the $p$-th power map induces an isomorphism
\[
H^{p^i}/H^{p^{i+1}} \xrightarrow{\sim} H^{p^{i+1}}/H^{p^{i+2}}.
\]
\end{itemize}
\end{definition}

The following gives an algebraic characterization of $p$-adic Lie groups:

\begin{theorem}{\textnormal{\cite[Theorem 8.32]{DdMS99}}}.
A group $G$ is a $p$-adic Lie group if and only if it is a topological group containing an open uniform pro-$p$ subgroup.
\end{theorem}

\begin{proof}
See sections 8.3 and 8.4 of \cite{DdMS99}.
\end{proof}

\begin{corollary}{\textnormal{\cite[Corollary 8.34(ii)]{DdMS99}.}}
If $G$ is a compact $p$-adic Lie group, then $G$ contains an open normal uniform subgroup of finite index.
\end{corollary}

By the corollary, if $G$ is a $p$-adic Lie group then it admits a normal uniform subgroup $H$. We get a filtration of $G$ by the normal subgroups $G(i):=H^{p^i}$,
\[
G\geq G(0)\geq G(1)\geq \cdots
\]
referred to as a \emph{Lie filtration}.

\begin{theorem}[Closed subgroup theorem]\label{thm:ClosedSubgroup}
Every closed subgroup $\Gamma$ of a $p$-adic Lie group $G$ is a $p$-adic Lie group. Additionally, if $H$ is an open uniform subgroup of $G$, then $\Gamma$ inherits a Lie filtration
\[
\Gamma\geq \Gamma(0)\geq \Gamma(1)\geq \cdots
\]
where $\Gamma(i)=H^{p^i}\cap \Gamma$ are open uniform subgroups of $\Gamma$.
\end{theorem}

\begin{proof}
See \cite[Theorem 9.6]{DdMS99}. 
\end{proof}

\begin{example}
Let $p$ be an odd prime. The group $G:=\GL_n(\Z_p)$ is a $p$-adic Lie group with open uniform subgroup $H:=1+p M_n(\Z_p)$, where $M_n(\Z_p)$ is the set of $n\times n$ matrices with entries in $\Z_p$. It admits a Lie filtration
\[
\GL_n(\Z_p)\geq G(0)\geq G(1)\geq \cdots
\]
where 
\[
G(i):=H^{p^i}=1+p^i M_n(\Z_p).
\]

By the closed subgroup theorem any closed subgroup $\Gamma\subseteq \GL_n(\Z_p)$, such as the image of a $p$-adic Galois representation, is a $p$-adic Lie group with Lie filtration
\[
\Gamma\geq \Gamma(i)\geq \Gamma(i+1)\geq \cdots
\]
where $\Gamma(i)=\Gamma\cap H^{p^i}$.

\end{example}

\section{Ramification Theory}\label{sec:Ramification}

In this section we review some basic facts from ramification theory. For additional details on ramification theory see \cite{Ser79}.

Let $K$ be a field complete with respect to a discrete valuation $v_K$. Let $R_K=\{a\in K : v_K(a)\geq 0\}$ denote the valuation ring of $K$. Suppose that $K$ is such that its residue field $k$ is algebraically closed. Set $p=\Char(k)$. Let $K^\times$ denote the unit group of $K$. Normalize $v_K$ so that its value group on $K^\times$ is $\Z$, and fix an extension of $v_K$ to a fixed separable closure $K^{\sep}$ of $K$. 

Let $L$ be a finite separable extension of $K$ with valuation ring $R_L$. Set $\Gamma_{L/K}=\Hom_K(L,K^{\sep})$. Note that if $L/K$ is Galois then $\Gamma_{L/K}=\Gal(L/K)$.

\subsection{Lower ramification filtration}

For each $x\in\R_{\geq -1}$ define the \emph{$x$-th lower ramification group} to be
\[
\Gamma_x:=\left\{\sigma\in \Gamma_{L/K} : v_K(\sigma(a)-a)\geq \frac{x+1}{[L:K]} \text{ for all } a\in R_L \right\}. 
\]
This is equivalent to 
\begin{align}\label{def:ramification_group}
\Gamma_x:=\left\{\sigma\in \Gamma_{L/K} : v_K(\sigma(\alpha)-\alpha)\geq \frac{x+1}{[L:K]} \right\} 
\end{align}
where $\alpha$ is such that $R_L=R_K[\alpha]$ as $R_K$ algebras \cite[IV \S 1 Lemma 1]{Ser79}. 

\begin{proposition}
\label{prop:RamFil}
The lower ramification groups form a decreasing, left-continuous chain of subgroups; namely,
\begin{itemize}
\item[(i)]  $\Gamma_{-1}=\Gamma_{L/K}$. 
\item[(ii)]  $\Gamma_x\subseteq \Gamma_y$ for all $x\geq y$.
\item[(iii)]  $\Gamma_x=\{1\}$ for all sufficiently large $x$.
\item[(iv)]  $\bigcap_{y<x}\Gamma_y = \Gamma_x$ for all $x$.
\end{itemize}
\end{proposition}

\begin{definition}
We call $x$ a \emph{break} in the (lower) ramification filtration if $\Gamma_x\neq \Gamma_{x+\epsilon}$ for all $\epsilon>0$. 
\end{definition}

When $L/K$ is a finite abelian extension, the Hasse-Arf theorem says that the breaks occur only at integers. However, if $L/K$ is not abelian the breaks may occur at other rational numbers. When $L/K$ is tamely ramified there is a unique break occurring at $x=0$.

\subsection{Upper ramification filtration} 

Define the \emph{Herbrand transition function}:
\begin{align*}
\phi_{L/K}: [0,\infty)&\to[0,\infty)\\
x&\mapsto  \int_0^x \frac{\# \Gamma_t }{[L:K]} dt.
\end{align*}

When $M$ is an extension of $K$ containing $L$ we have the identity
\begin{align}\label{eq:HerbrandID}
\phi_{M/K}=\phi_{L/K}\circ \phi_{M/L}.
\end{align}
Note that the Herbrand transition function is a strictly increasing continuous function, and therefore has an inverse $\psi_{L/K}$ defined on $[0,\infty)$. Sometimes we will instead write $\phi_{\Gamma}$ and $\psi_{\Gamma}$ for $\phi_{L/K}$ and $\psi_{L/K}$.

For each $x\in \R_{\geq 0}$ we define the \emph{$x$-th upper ramification group} to be 
\[
\Gamma^x:=\Gamma_{\psi(x)}.
\]
By Proposition \ref{prop:RamFil} the upper ramification groups form a filtration of $\Gamma$.

\begin{definition}
We call $x$ a \emph{break} in the (upper) ramification filtration if $\Gamma^x\neq \Gamma^{x+\epsilon}$ for all $\epsilon>0$. 
\end{definition}

By the Hasse-Arf theorem, all breaks in the upper ramification filtration occur at rational numbers.

The upper ramification filtration passes well to quotients in the following sense: For $M$ a finite Galois extension of $K$ containing $L$, set $G=\Gal(M/K)$ and $H=\Gal(M/L)$. Then for all $x\geq 0$,
\[
\Gamma^x=(G/H)^x=G^x H/H.
\] 
From this we may define an upper ramification filtration on infinite Galois extensions $M/K$:
\begin{align*}
\Gal(M/K)^x&:=\left\{ \sigma\in \Gal(M/K) :
  \begin{tabular}{@{}c@{}}
       \text{for all finite extensions $L/K$ contained in M}, \\
   $\sigma\in \Gamma_{L/K}^x \Gal(M/L)$
  \end{tabular}
\right\}.
\end{align*}
In this case we say $x$ is a break in the filtration if it occurs as a break in any finite quotient. 

Suppose now that $\Gamma=\Gal(L/K)$ is a $p$-adic Lie group, so that it has a Lie filtration
\[
\Gamma\supseteq \Gamma(1)\supseteq \Gamma(2)\supseteq \cdots.
\]
To avoid trivialities we assume $\Gamma(1)\neq 1$. We want to relate this Lie filtration to the ramification filtration. From the following proposition we get a weak relation, that will be useful to us later on:

\begin{proposition}{\cite[Lemma 3.4]{Sen72}}
\label{prop:ramLie}
For any open subgroup $H\leq \Gamma$,
\[
\Gamma^x\cap H=H^{\psi_{\Gamma/H}(x)}.
\]
\end{proposition}

\begin{proof}
We have
\[
\Gamma^x\cap H= \Gamma_{\psi_{\Gamma}(x)}\cap H=H_{\psi_\Gamma(x)}=H^{\phi_{H}\circ \psi_\Gamma(x)}=H^{\psi_{\Gamma/H}(x)},
\]
where the final equality follows from (\ref{eq:HerbrandID}).
\end{proof}

Applying the Proposition for $H=\Gamma(n)$, it follows that for all $n\geq 1$,
\[
\Gamma^x\cap \Gamma(n)=\Gamma(n)^{\psi_{\Gamma/\Gamma(n)}(x)}.
\]

\subsection{Newton Polygons}

Let us quickly recall the basic facts about Newton Polygons that will be used in this paper.

\begin{definition}[Newton Polygon]
Let $f(x)\in x^n+a_{n-1}x^{n-1}+\cdots+a_1x+a_0\in K[x]$. Then the Newton polygon of $f$, denoted $\mathfrak{N}(f)$, is the lower convex hull of the set of points $(i,v_K(a_i))$ in $\R^2$.
\end{definition}

\begin{proposition}[The theorem of the Newton Polygon]
\label{prop:NP}
Suppose that $f$ is separable over $K$. For each side $s_i$ of the Newton polygon of $f$, let $\lambda_i$ denote the slope of $s_i$ and $\mu_i$ denote the length of the projection of $s_i$ to the $x$-axis. If $\lambda_i<\infty$, then there are exactly $\mu_i$ roots of $f$ with valuation $-\lambda_i$.
\end{proposition}

\subsection{Generalized Artin--Schreier Extensions}
Let $q=p^r$ be a prime power and assume $\F_q\subset K$. Let $a\in K$ be such that $v_K(a)<0$ and $p\nmid v_K(a)$. These assumptions ensure that $a$ is not contained in the set $\{b^p-b : b\in K\}$, and therefore the polynomial $x^q-x-a\in K[x]$ is irreducible by \cite[Lemma 1.3]{GS91}. Such a polynomial is called a \emph{generalized Artin--Schreier polynomial}. If $\alpha$ is a root of a generalized Artin--Schreier polynomial (in an unspecified algebraic closure of $K$), then we call the extension $K(\alpha)$ a \emph{generalized Artin--Schreier extension}. Note that the roots of $x^q-x-a$ are $\alpha+c$ for $c\in \F_q\subset K$ and thus $K(\alpha)$ is the splitting field of $x^q-x-a$. Additionally, it is known that $\Gal(K(\alpha)/K)\cong \F_q^+$ is an elementary abelian $p$-extension of degree $q$ \cite[Porposition 1.2]{GS91}. 
 
The following proposition generalizes a well known result of Hasse \cite{Has35} (see also \cite[IV \S 2 Exercise 5]{Ser67}).

\begin{proposition}\label{prop:Artin-Schreier_breaks}
Let $L/K$ be the generalized Artin--Schreier extension obtained by adjoining a root of $x^q-x-a$. Then $\Gal(L/K)$ has a unique upper and lower ramification break at $-v_K(a)$.
\end{proposition}

\begin{proof}
We first show that $L/K$ is totally ramified. Let $\alpha$ be a root of $x^q-x-a$,  let $e(L/K)$ be the ramification index of the extension $L/K$, and let $v_L$ be a discrete valuation on $L$ normalized so that its value group on $L^\times$ is $\Z$. We have that
\begin{equation*}
e(L/K)v_K(a) = v_L(a) = v_L(\alpha^q-\alpha) = \min\{q v_L(\alpha), v_L(\alpha)\}.
\end{equation*}
Since we have assumed $v_K(a)<0$, it follows that $\min\{q v_L(\alpha), v_L(\alpha)\}<0$. Therefore $v_L(\alpha)<0$, and so
\begin{equation}\label{eq:valpha}
e(L/K) v_K(a)=\min\{q v_L(\alpha), v_L(\alpha)\}= q v_L(\alpha).
\end{equation}
Since we have assumed that $p\nmid v_K(a)$, this implies $q|e(L/K)$. As $e(L/K)\leq [L:K]=q$, we conclude that $e(L/K)=q$. This shows that the extension $L/K$ is a totally ramified.

Since $e(L/K)=q$, equation (\ref{eq:valpha}) implies that the valuation of $\alpha$ is \[
v_L(\alpha)=\frac{e(L/K)\cdot v_K(a)}{q}=v_K(a).
\]
 Note that for each $c\in \F_q$ we have that $\alpha+c$ is also a root of $x^q-x-a$. Let $\sigma_c$ be the element of the Galois group $\Gal(L/K)$ which sends $\alpha$ to $\alpha+c$. 

Since $p\nmid v_K(a)$, by the extended Euclidean algorithm there exist integers $s,t$ with $s>0$ and $1=sv_K(a)+tq$. Letting $\pi_K$ be a uniformizer for the valuation ring $R_K$, we have that $\pi_L:= \alpha^{s}\pi_K^t$ has valuation $v_L(\pi_L)=1$, and is therefore a uniformizer for $R_L$. Observe that 
\[
\frac{\sigma_c(\pi_L)}{\pi_L}=\frac{(\alpha+c)^s \pi_K^t}{\alpha^s \pi_K^t}=1+\frac{cs}{\alpha}+\cdots+\frac{c^s}{\alpha^s}.
\]
Therefore, for $c\in \F_q^\times$ we have
\begin{align}\label{eq:tau-break}
v_K(\sigma_c(\pi_L)-\pi_L)=v_K\left(\pi_L (cs/\alpha+\cdots+c^s/\alpha^s)\right)=\frac{1}{q}-v_K(\alpha)=\frac{1-v_L(\alpha)}{q}.
\end{align}
It now follows from the definition of the ramification groups (\ref{def:ramification_group}) that $L/K$ has a unique lower ramification break at $-v_L(\alpha)=-v_K(a)$. It is easy to check that this corresponds to a unique break in the upper ramification filtration at the same point.
\end{proof}

\begin{remark}
The extension $L/K$ in Proposition \ref{prop:Artin-Schreier_breaks} is an example of a \emph{one-dimensional elementary abelian extension} (see \cite{Eld09}). For more refined ramification data of one-dimensional elementary abelian extensions see work of Elder and Keating \cite{Eld09,EK20}.
\end{remark}

\section{Galois Representations of $p$-Divisible Groups in Characteristic $p$}
\label{sec:p-divisible groups}

Let $K$ be a field of characteristic $p>0$ and let $G=(G_v)$ be a $p$-divisible group over $K$. Let $\lambda=r/s\in [0,1]\cap \Q$ be a rational number in lowest terms. Let $G_\lambda$ denote the $p$-divisible group of slope $\lambda$, defined to be the unique (up to isomorphism) $p$-divisible group over $\F_p$ with Dieudonn\'e module 
\[
D(G_\lambda)=\Z_p[F,V]/(F^{s-r}=V^r,\ FV=VF=p).
\]
All endomorphisms of $G_\lambda$ are defined over $\F_{p^s}$  and
\[
\End_{\F_{p^s}}(G_\lambda)\cong \mathcal{O}_{G_\lambda}
\]
where $\mathcal{O}_{G_\lambda}$ is an order in $D_\lambda$, the central division algebra over $\Q_p$ with Hasse invariant $\lambda$. By the Dieudonn\'e--Manin classification Theorem~\cite{Die57,Man63}, the $p$-divisible group $G$ is isogenous over $K^{\sep}$ to
\[
\prod_{\lambda\in \Q} G_\lambda^{d_\lambda}
\]
where the $d_\lambda$ are integers uniquely determined by $G$, all but finitely many of which are zero.

\begin{definition}[Generalized Tate Module, \cite{Gro79}] Define the \textit{generalized Tate module}
\[
T^\lambda(G):=\Hom_{K^{\sep}}(G_\lambda,G)
\]
and corresponding $\Q_p$-vector space
\[
V^\lambda(G):=\Hom_{K^{\sep}}(G_\lambda,G)\otimes_{\Z_p}\Q_p.
\]
\end{definition}

Observe that $V^\lambda(G)$ is a right $D_\lambda$-module of dimension $d_\lambda$ and a left module of dimension $d_\lambda$ over the opposite algebra $D_\lambda^{op}$. There is a left action of the absolute Galois group $\mathfrak{G}_K=\Gal(K^{\sep}/K)$ on $V^\lambda(G)$ given by
\begin{align*}
\mathfrak{G}_K\times V^\lambda(G) &\to V^\lambda(G)\\
(\sigma, f)&\mapsto \sigma\circ f \circ \sigma^{-1}.
\end{align*}
When $K$ contains $\F_{p^s}$ this action is $D_\lambda^{op}$ linear and we thus get a representation
\[
\rho^\lambda: \mathfrak{G}_K\to \Aut_{D_\lambda^{op}}(V^\lambda(G))=\GL_{d_\lambda}(D_\lambda).
\]
When $K$ contains $\overline{\F}_p$ we get a representation
\[
\rho=\prod_{\lambda\in [0,1]\cap\Q} \rho^\lambda: \mathfrak{G}_K\to \prod_{\lambda\in [0,1]\cap\Q} \GL_{d_\lambda}(D_\lambda).
\]
For any ring $R$, let $R^\times$ denote the group of units of $R$. We define a determinant map on $\prod \GL_{d_{\lambda}}(D_\lambda)$ as follows:
\[
\det=\prod_{\lambda\in [0,1]\cap\Q} \Nm_{\lambda}: \prod_{\lambda\in [0,1]\cap\Q} \GL_{d_\lambda}(D_\lambda)\to \Q_p^\times,
\]
where $\Nm_\lambda$ is the reduced norm in $\Mat_{d_{\lambda}}(D_\lambda)$ over $\Q_p$. Composing with $\rho$ we get a $p$-adic Galois character $\chi:=\det\circ \rho$. 

\begin{theorem}\cite[Theorem 2.7]{Gro79}\label{thm:G2.7}
If $G$ is a $p$-divisible group defined over $R$, then $\chi=1$ in $\Hom(\mathfrak{G}_K, \Q_p^\times)$. 
\end{theorem}

Later we will apply Theorem \ref{thm:G2.7} in the case that $G$ has dimension 1. In this case $\rho=\rho^{0/1}\oplus \rho^{1/g}$, where $g$ is the dimension of the generic fiber of $G$. Theorem \ref{thm:G2.7} then says that $\Nm_{1/g}\circ \rho^{1/g}$ is the inverse of $\det\circ \rho^{0/1}$ in $\Q_p^\times$. In particular, the Galois characters $\det\circ \rho^{0/1}$ and $\Nm_{1/g}\circ \rho^{1/g}$ determine each other.

\section{Formal Groups and Formal Modules}
\label{sec:formal groups}

There is an equivalence of categories between connected $p$-divisible groups of (finite) height $h$ over $R$ and formal groups of height $h$ over $R$ \cite[Proposition 1]{Tat67}. In this section we introduce formal groups and formal modules. We will later see how the theory of formal modules can be used to give a more explicit description of the Galois representation $\rho$ introduced in the previous section.

\begin{definition}[Formal Group]
A \emph{(commutative one-parameter) formal group} $\mathcal{F}$ over a ring $R$ is a power series $\mathcal{F}(X,Y)\in R[[x,y]]$ satisfying the following properties:
\begin{itemize}
\item[(i)] $\mathcal{F}(X,0)=X$ and $\mathcal{F}(0,Y)=Y$.
\item[(ii)] $\mathcal{F}(X,F(Y,Z))=\mathcal{F}(\mathcal{F}(X,Y),Z)$.
\item[(iii)] $\mathcal{F}(X,Y)=\mathcal{F}(Y,X)$.
\end{itemize}
\end{definition}

\begin{definition}[Morphisms of Formal Groups]
Let $\mathcal{F}$ and $\mathcal{G}$ be formal groups over $R$. A \emph{homomorphism} from $\mathcal{F}$ to $\mathcal{G}$ is a power series $f\in R[[T]]$ with no constant term satisfying
\[
f(\mathcal{F}(X,Y))=\mathcal{G}(f(X),f(Y)).
\]
 The formal groups $\mathcal{F}$ and $\mathcal{G}$ are isomorphic over $R$ if there are homomorphisms $f:\mathcal{F}\to \mathcal{G}$ and $g:\mathcal{G}\to \mathcal{F}$ defined over $R$ satisfying $f(g(T))=g(f(T))=T$.
\end{definition}

\begin{definition}[Formal Module]
Let $A$ be a ring and assume $R$ is an $A$-algebra. A \emph{(one-dimensional one-parameter) formal $A$-module} is a (one-parameter) formal group $\mathcal{F}$ over $R$ equipped with a ring homomorphism $[\cdot]_\mathcal{F}:A\to \End_R(\mathcal{F})$ such that for each $a\in A$ there is an endomorphism $[a]_\mathcal{F}$ of $\mathcal{F}$ satisfying $[a]_\mathcal{F}(X)= aX+ \textnormal{(higher degree terms)}$.
\end{definition}

\begin{definition}[Morphisms of Formal Modules]
Let $\mathcal{F}$ and $\mathcal{G}$ be formal $A$-modules over $R$. A \emph{homomorphism} from $\mathcal{F}$ to $\mathcal{G}$ is a homomorphism of formal groups $f:\mathcal{F}\to\mathcal{G}$ satisfying 
\[
f\circ [a]_\mathcal{F}=[a]_\mathcal{G} \circ f
\]
for all $a\in A$. 
\end{definition}

Let $\Hom_R(\mathcal{F},\mathcal{G})$ denote the set of $A$-module $R$-homomorphisms over from $\mathcal{F}$ to $\mathcal{G}$. Set $\End_R(\mathcal{F})=\Hom_R(\mathcal{F},\mathcal{F})$.

\subsection{Formal $A$-modules over rings of characteristic $p$}

We consider the case when $A$ is the ring of integers of a finite extension of $\Q_p$ and $R=K$ is a field of characteristic $p$ (e.g., $A=\Z_p$ and $R=\overline{\F}_p((t))$\ ). Let $\pi$ be a uniformizer for $A$ and set $q=\#(A/\pi A)$.

\begin{definition}[Height]
Let $\mathcal{F}$ and $\mathcal{G}$ be formal $A$-modules over $K$ and let $\phi\in \Hom_R(\mathcal{F},\mathcal{G})$. The height of $\phi$, denoted $\Ht(\phi)$, is the largest integer $h$ such that
\[
\phi(T)=f(T^{q^h})
\] 
for some power series $f\in K[[T]]$. If there is no such $h$, then set $\Ht(\phi)=\infty$.

The \emph{height} of of a formal $A$-module $\mathcal{F}$ over $K$, denoted $\Ht(\mathcal{F})$, is $\Ht([\pi]_\mathcal{F})$.
\end{definition}

\begin{remark}
Though we won't consider formal group height in this article, we remark that it differs from formal module height by a factor of $[A:\Z_p]$. 
\end{remark}

\begin{proposition}
Let $f:\mathcal{F}\to \mathcal{G}$ and $g:\mathcal{G}\to \mathcal{H}$ be homomorphisms of formal $A$-modules defined over $R$. Then $\Ht(g\circ f)=\Ht(f)+\Ht(g)$.
\end{proposition}

\begin{proof}
See \cite[IV.7 Proposition 7.3]{Sil09}.
\end{proof}

\begin{proposition}
\label{prop:height characterization}
If $K$ is separably closed, then all formal $A$-modules of finite height $h$ are isomorphic. Moreover, any such formal $A$-module has endomorphism ring isomorphic to the ring of integers of the central division algebra over $K$ of Hasse invariant $1/h$.
\end{proposition}

\begin{proof}
See \cite[Proposition 1.7]{Dri74}
\end{proof}

\subsection{$A$-typical formal modules}

Set $A[v]:=A[v_1,v_2,\dots]$. Let $f(x)$ be the power series with coefficients in $A[v]\otimes \Frac(A)=\Frac(A)[v_1,v_2,\dots]$ uniquely determined by the functional equation
\[
f(x) = x + \sum_{i=1}^\infty \frac{v_i}{\pi} f^{(q^i)}(x^{q^i})
\]
where $f^{(q^i)}(x)$ denotes the power series obtained from $f(x)$ by replacing each $v_j$ by $v_j^{q^i}$. More explicitly
\[
f(x)=\sum_{i=0}^\infty b_i x^{q^i}
\]
where the coefficients are defined recursively:
\begin{align*}
b_0&= 1\\
b_i &= \frac{b_0v_i +b_1v^q_{i-1}+b_2 v^{q^2}_{i-1}+\cdots + b_{i-1}v_1^{q^{i-1}}}{\pi}
\end{align*}
Induction on $i$ shows that $\pi^i b_i\in A[v]$ for all $i$. 

By Hazewinkel's functional equation lemma \cite[Lemma 4.2]{Haz79}, the power series
\begin{align*}
F_V(x,y):=f^{-1}(f(x)+f(y))\\
[a]_\mathcal{F}(x):= f^{-1}(a f(x))\ \text{ for all } a\in A
\end{align*}
gives rise to a formal $A$-module, which we shall denote $\mathcal{F}_V$. One can show that the multiplication-by-$\pi$ map satisfies the congruence
\begin{align}\label{eq:piAtypical}
[\pi]_\mathcal{F}(x)\equiv v_i x^{q^i} \pmod{\pi,v_1,\dots,v_{i-1}, x^{q^i+1}},
\end{align}
where we are considering $(\pi,v_1,\dots,v_{i-1},x^{q^i+1})$ as an ideal in the ring of power series over $A[v]$.

\begin{definition}[$A$-Typical Formal $A$-Module]
A formal $A$-module over $R$ is called \emph{$A$-typical} if it is the specialization of $\mathcal{F}_V$ with respect to an $A$-algebra homomorphism $\phi:A[v]\to R$. If $\phi(v)$ denotes the sequence $(\phi(v_1), \phi(v_2),\dots)$ in $R$, we shall denote the specialization of $\mathcal{F}_V$ by $\mathcal{F}_{\phi(v)}$. Note that specifying an $A$-algebra homomorphism $A[v]\to R$ is the same as choosing an $r_i\in R$ for each $v_i$.
\end{definition}

\begin{theorem}\cite[Corollary 2.10]{Haz79}
\label{thm:A-typical}
Every formal $A$-module over $R$ is isomorphic to an $A$-typical formal $A$-module over $R$.
\end{theorem}

\begin{example}[Honda Formal Modules]
\label{ex:F1/g}
Let $\mathcal{F}_{1/h}$ denote the $A$-typical formal $A$-module over $\F_q\cong A/\pi A$  with respect to the homomorphism $\phi: A[v]\to \F_q$ which sends $v_h$ to $1$ and $v_i$ to $0$ for all $i\neq h$. One can show that $\mathcal{F}_{1/h}$ is of height $h$ and that
\[
[\pi]_{\mathcal{F}_{1/h}}(T)=T^{q^h}.
\]
All endomorphisms of $\mathcal{F}_{1/h}$ are defined over $\F_{q^h}$ and
\[
\End_{\F_{q^h}}(\mathcal{F}_{1/h})=B_{1/h}
\]
where $B_{1/h}$ is the ring of integers of the central division algebra defined over the field $A\otimes_{\Z_p} \Q_p$ with Hasse invariant $1/h$. 

When $h=1$, the formal group underlying $\mathcal{F}_{1/1}$  is the formal multiplicative group $\hat{\mathbb{G}}_m$ given by the formal group law
\[
X+Y+XY=(1+X)(1+Y)-1.
\]
\end{example}

\subsection{From $p$-divisible groups to formal Modules}
\label{sec:p-div to formal}

Let $\mathcal{F}$ be a formal group of height $h$ over $R$. Let $v\geq 1$ be an integer and $I_v$ denote the ideal of $R[[T]]$ generated by the endomorphism $[p^v]_\mathcal{F}(T)\in \End_{R}(\mathcal{F})$. Then, $G_v:=\Spec(R[[T]]/I_v)$ is a finite commutative group scheme of order $p^{vh}$. These $G_v$ combine to determine a connected 1-dimensional $p$-divisible group $G$. One can show that $\mathcal{F}$ and $G$ determine each other. See \cite[Proposition 1]{Tat67}.

Let $\mathcal{F}$ be a formal $A$-module over $R$. Let $G_{1/h}$ be the $p$-divisible group corresponding to the formal $A$-module $\mathcal{F}_{1/h}$ from Example \ref{ex:F1/g}. Note that the special fiber of $\mathcal{F}$ is isomorphic to $\mathcal{F}_{1/s}$ over $k$. Let  $G_{0/1}$ denote the constant \'etale $A$-module $(A\otimes_{\Z_p} \Q_p)/A$. Let $\mathcal{F}_K$ denote the generic fiber of $\mathcal{F}$. Set $g=\Ht(\mathcal{F}_K)$. Let $G_K$ denote the generic fiber of the p-divisible group $G$. Since the generic fiber $\mathcal{F}_K$ is $1$-dimensional, there is a $K^{\sep}$-isogeny 
\[
G_K \to G_{1/g} \times (G_{0/1})^d,
\]
where $d=s-g$.

Define the slope $1/g$-Tate module
\[
T^{1/g}(\mathcal{F}):=\Hom_{K^{\sep}}(\mathcal{F}_{1/g}, \mathcal{F}_K),
\]
which is of rank 1 over $\End(\mathcal{F}_{1/g})=B_{1/g}$. Define the slope $0$-Tate module 
\[
T^{0/1}(G):=\Hom_{K^{\sep}}(G_{0/1},G_K),
\]
which is of rank $d$ over $A$. These give rise to Galois representations
\[
\rho^{1/g}: \mathfrak{G}_K\to B_{1/g}^\times
\]
and
\[
\rho^{0/1}: \mathfrak{G}_K \to \GL_d(A)
\]
corresponding to the $\rho^\lambda$ in Section \ref{sec:p-divisible groups}.

\section{Towers Arising From $\rho^{0/1}$}
\label{sec:0/1-towers}

Let $\mathcal{F}$ be a 1-dimensional formal $A$-module over $R$ with special fiber of height $s$, and generic fiber of height $g$ with $1\leq g < s$. Set $d=s-g$.
Let $G$ be the corresponding 1-dimensional $p$-divisible group over $R$ (see \S \ref{sec:p-div to formal}).
 In this section we study $p$-adic Lie towers of local fields associated to the Galois representation $\rho^{0/1}$.

\subsection{The torsion tower and its subtowers}

The representation $\rho^{0/1}$ arises from the action of the absolute Galois group $\mathfrak{G}_K$ on the Tate module $T(G)=\varprojlim G[\pi^n](K^{\sep})$. As $T(G)$ is a free $A$-module of rank $d$, we may view $\rho^{0/1}$ as a homomorphism
\[
\rho^{0/1}:\mathfrak{G}_K\to \GL_d(A).
\]
Let $M$ denote the ring of $d\times d$ matrices in $A$. Note that $M^\times=\GL_d(A)$. The Lie filtration on $M^\times$,
\[
M^\times\supseteq 1+\pi M\supseteq 1+\pi^2 M\supseteq\cdots
\]
gives rise to a tower of fields as follows: Let $\delta_n$ be the reduction map
\[
\delta_n:M^\times\to \frac{M^\times}{1+\pi^{n+1} M} \cong \left(\frac{M}{\pi^{n+1}M}\right)^\times
\] 
and let $\mathfrak{H}_n\subseteq \mathfrak{G}_K$ be the kernel of $\delta_n\circ \rho^{1/0}$. Letting $K_n\subset K^{\sep}$ denote the fixed field of $\mathfrak{H}_n$ we get a $p$-adic Lie tower of local fields:
\[
K=K_{-1}\subseteq K_0 \subseteq K_1\subseteq K_2\subseteq \cdots \subseteq K_\infty \subseteq K^{\sep}
\]
where $K_\infty$ is the compositum of the $K_n$. We refer to this tower as the \emph{torsion tower}, since the torsion representation $\rho_n$,
\[
\rho_n:\mathfrak{G}_K\to \Aut_{A/\pi^{n+1} A}(G[\pi^{n+1}](K^{\sep}))\cong \left(\frac{M}{\pi^{n+1}M}\right)^\times
\]
is canonically identified with $\delta_n\circ \rho^{0/1}$, compatibly with change in $n$. In particular, the $K_n$ are the $\pi^{n+1}$-division fields of $G(K^{\sep})$.

For each $y=(y_0,y_1,\dots)\in T(G)=\varprojlim G[\pi^n](K^{\sep})$ we get a tower of local fields
\[
K\subseteq K(y_0)\subseteq K(y_1)\subseteq \cdots \subseteq K^{\sep}.
\]
This tower can be viewed as a \emph{subtower} of the torsion tower, in the sense that $K(y_{n-1})\subseteq K_n$ for all $n$. Additionally, each $K_n$ is the compositum of the $K(y_{n-1})$ as $y$ ranges over all elements of $T(G)$.

\begin{remark}
Katz used properties of the $\{K(y_n)\}_n$ towers in his proof of Igusa's Theorem~\cite[Theorem 4.3]{Kat73}.
\end{remark}

\subsection{Properties of the subtowers}

We study these towers by considering the multiplication-by-$\pi$ map on $G$. By Example \ref{ex:F1/g} and the deformation theory of formal modules \cite[Theorem 22.4.4]{Haz78}, we may replace $\mathcal{F}$ by an isomorphic formal $A$-module, so that the multiplication-by-$\pi$ map, which is a priori a power series, is in-fact a polynomial: 
\[
[\pi]_\mathcal{F}(x)=a_1 x^{q^g}+a_2 x^{q^{g+1}}+\cdots+a_{d} x^{q^{g+d-1}}+x^{q^s}\in R[x]
\]
with each $a_i\in \pi_R R$ and $a_1\neq 0$. Roots of this polynomial correspond to the $\pi$-torsion points of $G(K^{\sep})$. If we insist on these $\pi$-torsion points belonging to $G[\pi](K^{\sep})$ then we may instead consider the separable polynomial $V(x)=[\pi]_\mathcal{F}(x^{1/q^g})$: 
\[
V(x)= a_1 x + a_2 x^q + \cdots + a_d x^{q^{d-1}} + x^{q^d}.
\]

\begin{remark}
 When $g=1$, $V$ corresponds to verschiebung for $G$.
\end{remark}

Twisting the coefficients of $V$ by $q^{ig}$-power Frobenius, we define
\[
V^{(q^{ig})}(x)=a_1^{q^{ig}}x+a_2^{q^{ig}} x^{q}+\cdots+a_{d}^{q^{ig}} x^{q^{d-1}}+x^{q^d}.
\]

Then specifying an element of the Tate module $T(G)=\varprojlim G[\pi^n](K^{\sep})$ amounts to solving the system of equations
\begin{align*}
V(y_0) &=0,\\
V^{(q^g)}(y_1) &= y_0,\\
V^{(q^{2g})}(y_2) &= y_1,\\
V^{(q^{3g})}(y_3) &= y_2,\\
 &\ \vdots
\end{align*}
If $y_0\neq 0$ then each $y_n$ will be a point of order $\pi^{n+1}$.

\begin{remark}
One can arrive at the polynomial given for $[\pi]_\mathcal{F}(x)$ above by more elementary means, avoiding deformation theory of formal modules. To see this, choose a model for $\mathcal{F}$ such that the multiplication-by-$\pi$ is
\[
[\pi]_\mathcal{F}(x)=a_1x^{q^g}+a_2x^{2q^g}+\cdots
\]
with $a_1\neq 0$ and $a_i\in \pi_R R$ for all $i\neq q^d$. Applying the $p$-adic Weierstrass preparation theorem to the power series $[\pi]_\mathcal{F}(x)$ gives a polynomial coinciding with the polynomial given for $[\pi]_\mathcal{F}(x)$ above. Though this method does not show $[\pi]_\mathcal{F}(x)$ is a polynomial, it does give a polynomial whose roots coincide with the roots of $[\pi]_\mathcal{F}(x)$ in $\{x\in K^{\sep}: v_K(x)>0\}$. This idea was used in the proof of \cite[Lemma 4.14]{Gro79}. 
\end{remark}

The following lemma generalizes \cite[Theorem 3.3]{Cha00}:

\begin{proposition}
\label{lem:chaiNP}
There exists a positive integer $m$, depending only on the valuations $v_K(a_j)$, such that for each $y=(y_0,y_1,\dots)\in T(G)$ with $y_0\neq 0$, $v_K(y_n)=q^{-d}\cdot v_K(y_{n-1})$ for all $n\geq m$. Additionally $K(y_n)/K(y_{n-1})$ is a totally ramified extension of degree $q^{d}$ for each $n\geq m$.
\end{proposition}

\begin{proof}
We have that $V(y_0)=0$ and $V^{(q^{ig})}(y_i)=y_{i-1}$ for all $i>0$. As each $V^{(q^{ig})}$ is of degree $q^d$ we have that $q^d v_K(y_{i+1})\leq v_K(y_i)$ for all $i\geq 0$. Also for each $i$, $y_i$ is a root of the degree $q^{id}$ monic polynomial $V\circ V^{(q^g)}\circ \cdots\circ V^{(q^{ig})}$, and therefore the denominator of $v_K(y_i)$ is at most $q^{id}$ when written in lowest terms.

Consider the Newton polygon of $V(x)$. Note that $(1, v_K(a_1))$ is the point of highest $y$-value on the (lower) boundary of $\mathfrak{N}(V(x))$. It then follows from Proposition \ref{prop:NP} that the valuation of any of the roots of $V(x)$ is at most $v_K(a_1)$. In particular $v_K(y_0)\leq v_K(a_1)$. Since $q^d v_K(y_{i+1})\leq v_K(y_{i})$ for all $i\geq 0$ it follows that $q^{id}v_K(y_i)\leq v_K(a_1)$ for all $i$. In particular, the numerator of $v_K(y_i)$, when written in lowest terms, is at most $v_K(a_1)$.

Now we consider the Newton polygons of the polynomials $V^{(q^{ig})}(x)-y_{i-1}$. We have just seen that $v_K(y_i)\leq v_K(a_1)$ for all $i$. On the other hand, for each $j$, $v_K(a_j^{q^{ig}})=q^{ig} v_K(a_j)\geq q^{ig}$ tends to infinity as $i\to\infty$. From this we see that there exists an $m'$, depending only on the valuations of the $a_j$, such that for all $i\geq m'$ the boundary of $\mathfrak{N}(V^{(q^{ig})}(x)-y_{i-1})$ is the line segment between $(0, v_K(y_{i-1}))$ and $(q^{d},0)$.  It follows from Proposition \ref{prop:NP} that $v_K(y_i)=q^{-d}\cdot v_K(y_{i-1})$ for all $i\geq m'$. 

As the numerator of $v_K(y_i)$ is at most $v_K(a_1)$ for all $i$, the largest power of $q^d$ that can divide the numerator of $v_K(y_i)$ is less than or equal to $\lceil \log_{q^d}(v_K(a_1))\rceil$. Set $m:=m'+\lceil \log_{q^d}(v_K(a_1))\rceil+1$. By the previous paragraph, for $n\geq m$ the denominator of $v_K(y_n)$ is exactly $q^d$ times greater than the denominator of $v_K(y_{n-1})$. This implies that if $n\geq m$ then the ramification index $e(K(y_n)/K(y_{n-1}))$ is at least $q^d$. On the other hand, since $\deg(V^{(q^{ig})}(x)-y_{i-1})=q^d$, the index $[K(y_i):K(y_{i-1})]$ is at most $q^{d}$. Since $e(K(y_i)/K(y_{i-1}))$ divides $[K(y_i):K(y_{i-1})]$, it follows that
\[
e(K(y_n)/K(y_{n-1}))=[K(y_n):K(y_{n-1})]=q^{d}
\]
for all $n\geq m$. 
\end{proof}

The following irreducibility theorem generalizes \cite[Theorem 3.5]{Cha00}:

\begin{theorem}
\label{thm:irreducible}
Every non-zero orbit of the action of the monodromy group $\rho^{0/1}(\mathfrak{G}_K)$ on the Tate module $T(G)$ is open. In particular, the representation $\rho^{0/1}$ is irreducible.
\end{theorem}

\begin{proof}
Let $Z$ be a non-zero $\rho^{0/1}(\mathfrak{G}_K)$-orbit of $T(G)$. Note that the subset
\[
 U:=T(G)-\pi T(G)\subset T(G)
\]
is stable under the linear action of $\rho^{0/1}(\mathfrak{G}_K)$. It follows that the subsets $\pi^i U$ are $\rho^{0/1}(\mathfrak{G}_K)$-stable for all $i\geq 1$. Let $m\in\Z_{\geq 1}$ be such that $Z\subseteq \pi^m U$.

Let $(y_i)$ be a sequence of elements in $K^{\sep}$ satisfying the conditions in Lemma \ref{lem:chaiNP}. By Lemma \ref{lem:chaiNP} the orbit of $y_n$ under $\Gal(K^{\sep}/K(y_{n-1}))$ has $q^{d}$ elements for each $n\geq N$. As $y_0\neq 0$ and each $y_{i-1}$ has order $\pi^i$, the sequence $(y_i)$ corresponds to an element $y$ of $U$. The coset 
\begin{align}\label{eq:orbit}
\left(\Gal(K^{\sep}/K(y_{n-1}))\cdot y + \pi^{n+1} T(G)\right)\big/ \pi^{n+1} T(G)
\end{align}
is in bijection with the orbit of $y_n$ under $\Gal(K^{\sep}/K(y_{n-1}))$, and thus contains $q^{d}$ elements.

We claim that for each $z=(z_i)\in Z$,
\begin{align}\label{eq:orbit2}
((z+\pi^{m+n} T(G))\cap Z) + \pi^{m+n+1}T(G) = z+\pi^{m+n}T(G)
\end{align}
for all $n>N$. The left hand side corresponds to $u=(u_i)\in \pi^m U$ such that there exists a $z'=(z'_i)\in Z$ with $u_i=z'_i$ for all $i\leq m+n$ and $z'_i=z_i$ for all $i<m+n$. The right hand side corresponds to all elements $u=(u_i)$ of $\pi^m U$ for which $z_i=u_i$ for all $i<m+n$. In particular, the left hand side of (\ref{eq:orbit2}) is contained in the right hand side. Therefore, to prove equality it suffices to show that the two sets have the same cardinality modulo $\pi^{m+n+1}T(G)$. The left hand side modulo $\pi^{m+n+1}T(G)$ has cardinality $q^d$, since it follows from (\ref{eq:orbit}) that there are $q^d$ possibilities for $z'_{m+n}$ such that $(z'_0,\dots,z'_{m+n})=(z_0,\dots,z_{m+n-1},z'_{m+n})$. The right hand side modulo $\pi^{m+n+1}T(G)$ also has cardinality $q^d$, because each of its elements can be written as $(z_0,\dots,z_{m+n-1},t)$ with $t\in G[\pi^{m+n}]^{et}$, and there are precisely $q^d$ possibilities for $t\in G[\pi^{m+n}]^{et}$ since $T(G)$ is a free $A$-module of rank $d$.

The equality (\ref{eq:orbit2}) implies that if $u\in \pi^m U$ and there exists a $z\in Z$ for which $z\equiv u \pmod{\pi^{m+n+1}T(G)}$ for any $n\geq N$, then $u\in Z$. In particular, $Z$ contains a translate of $\pi^{m+N+1} T(G)$. Thus $Z$ has dimension $\dim(Z)=\dim(T(G))$ as a $p$-adic analytic subvariety of $T(G)$, i.e., $Z$ is open in $T(G)$.
\end{proof}

Since $T(G)$ is a finite free $A$-module, $T(G)-\pi T(G)$ is compact. From this observation we obtain the following corollary:

\begin{corollary}
\label{cor:compact}
There are only finitely many $\rho^{0/1}(\mathfrak{G}_K)$-orbits in $T(G)-\pi T(G)$.
\end{corollary}


\section{Towers Arising From $\rho^{1/g}$}
\label{sec:1/gTowers}

In this section we introduce and study certain $p$-adic Lie towers of local fields associated to the Galois representation $\rho^{1/g}$. 

\subsection{The $1/g$-tower}

Recall that the the $p$-divisible group $G$ gives rise to a continuous homomorphism
\[
\rho^{1/g}: \mathfrak{G}_K\to B_{1/g}^\times,
\]
where $B_{1/g}$ is the ring of integers of the division algebra over the field $A\otimes_{\Z_p} \Q_p$ with Hasse invariant $1/g$. Let $\pi_{1/g}$ be a uniformizer for $B_{1/g}$. Then there is a $\pi_{1/g}$-filtration of $B_{1/g}^\times$
\[
B_{1/g}^\times \supset 1+\pi_{1/g} B_{1/g} \supset 1+\pi_{1/g}^2 B_{1/g} \supset \cdots.
\]

This filtration gives rise to a $p$-adic Lie tower of local fields as follows: We appropriate our notation from earlier, letting $\delta_n$ be the reduction map
\[
\delta_n:B_{1/g}^\times \to \frac{B_{1/g}^\times}{1+\pi_{1/g}^{n+1} B_{1/g}}\cong \left(\frac{B_{1/g}}{\pi_{1/g}^{n+1} B_{1/g}}\right)^\times,
\]
and $\mathfrak{H}_n\subseteq \mathfrak{G}_K$ be the kernel of the composition $\rho_n:=\delta_n\circ\rho_{1/g}$. Denoting the fixed field of $\mathfrak{H}_n$ by $K_n\subseteq K^{\sep}$ we obtain a tower of fields,
\[
K=K_{-1}\subseteq K_0 \subseteq K_1 \subseteq K_2 \subseteq \cdots \subseteq K_\infty \subseteq K^{\sep},
\]
where $K_\infty$ is the compositum of the $K_n$ and where $\Gal(K_n/K)\cong \mathfrak{G}_K/\mathfrak{H}_n$. We refer to this tower as the \emph{$1/g$-tower of $G$}. 

\subsection{Properties of the $1/g$-tower} 

In analogy to the torsion tower, we consider subtowers which arise from elements of the generalized Tate module $T^{1/g}(\mathcal{F})=\Hom_{K^{\sep}}(\mathcal{F}_{1/g},\mathcal{F}_{K})$. We will focus on invertible elements of $T^{1/g}(\mathcal{F})$, i.e., elements corresponding to $K^{\sep}$-isomorphisms. Fix a $K^{\sep}$-isomorphism $\varphi:\mathcal{F}_{1/g}\to \mathcal{F}$. We write the inverse as
\[
\varphi^{-1}(T)=c_1T+c_2T^2+\cdots\in K^{\sep}[[T]]
\]

For $\sigma\in \mathfrak{G}_K$ let $\varphi^{\sigma}$ denote the morphism of formal $A$-modules obtained by the action of $\sigma$ on the coefficients of the power-series for $\varphi$. Noting that $\mathfrak{G}_K$ acts trivially on the coefficients of $\mathcal{F}_{1/g}$, we can write $\rho^{1/g}$ in terms of $\varphi$ as follows:
\begin{align*}
\rho^{1/g}:\mathfrak{G}_K &\to \Aut(\mathcal{F}_{1/g})\cong B_{1/g}^\times\\
\sigma&\mapsto \varphi^{-1}\circ\varphi^{\sigma}.
\end{align*}

Recall that 
\[
[\pi]_\mathcal{F}(x)=a_1 x^{q^g}+a_2 x^{q^{g+1}}+\cdots+a_dx^{q^{g+d-1}}+x^{q^s}\in R[x]
\]
and
\[
[\pi]_{\mathcal{F}_{1/g}}(x)=x^{q^g}.
\]
Since $\varphi^{-1}$ is an isomorphism of formal modules from $\mathcal{F}$ to $\mathcal{F}_{1/g}$,
 \begin{align}\label{eq:FormalID}
\varphi^{-1} \circ [\pi]_{\mathcal{F}}(x)=[\pi]_{\mathcal{F}_{1/g}}\circ \varphi^{-1}(x)=c_1^{q^g}x^{q^g}+c_2^{q^g}x^{2q^g}+\cdots.
\end{align}
Fix a positive integer $n$. For each positive integer $t$ define the set
\[
S_t:=\left\{(i_1,\dots,i_t)\in \{0,1,\dots,d\}^t : q^{n+g}=\sum_{j=1}^t q^{i_j+g}\right\}.
\]
Set $a_{d+1}:=1$ and 
\[
b:=\sum_{t=1}^{q^n-1} c_t \sum_{(i_1,\dots,i_t)\in S_t} \prod_{j=1}^t a_{i_j+1}.
\]
Then, comparing coefficients of $x^{q^{n+g}}$ in (\ref{eq:FormalID}), we have that
\begin{align}\label{eq:FormalIDCOEF}
a_1^{q^n}c_{q^n}+b=c_{q^n}^{q^g}.
\end{align}

We will now study the tower
\[
K\subseteq K(c_1) \subseteq K(c_1, c_2) \subseteq K(c_1,c_2,c_3)\subseteq \cdots\subseteq K^{\sep}.
\]
Gross proved the following results about this tower:

\begin{proposition}{\cite[Lemma 4.2(1)]{Gro79}}
\label{prop:varphi property1}
The coefficients $c_i$ of $\varphi$ are integral in $K^{\sep}$. 
\end{proposition}

\begin{proposition}{\cite[Lemma 4.2(2)]{Gro79}}
\label{prop:varphi property2}
 If $j<q^n$ then $c_j\in K_{n-1}$ and $K_{n}=K_{n-1}(c_{q^n})$.
\end{proposition}

For each $n$ let $v_n$ denote the valuation on $K_n$ obtained by first extending $v_K$ and then re-normalizing so that the value group of $v_n$ on $K_{n-1}^\times$ is $\Z$. We now state our main result:

\begin{theorem}
\label{thm:1/gBreaks}
Suppose the image of $\rho^{1/g}$ is open in $B_{1/g}^\times$ so that there exists an $N$ such that $1+\pi_{1/g}^{n+1}B_{1/g}\subseteq \rho^{1/g}(\mathfrak{G}_K)$ for all $n\geq N$. Define
\[
W(n)=-v_{N}\left(a_1^{-\frac{q^{g+n}}{q^g-1}} b \right)/(q^g-1)-\frac{1}{q^g} \sum\limits_{j=N+1}^{n-1} v_{N}\left(a_1^{-\frac{q^{g+j}}{q^g-1}} b\right)
\]
Then, for $n>N$, $W(n)$ are the upper breaks in the upper ramification filtration of $\rho^{1/g}(\mathfrak{G}_{K_N})$ and
\[
\rho^{1/g}(\mathfrak{G}_{K_N})^{W(n)}=1+\pi_{1/g}^{n}B_{1/g}.
\]
\end{theorem}

\begin{remark}
Based on Galois representations in other settings one would expect, or at least hope, that the image of $\rho^{1/g}$ is always open. This is known in the case that $d=1$ \cite{Gro79} and in the case $g=1$ and $A=\Z_p$ \cite{Cha00}. Results of Achter and Norman imply that if $g<s$ and $g\neq 2$ then there exists a $p$-divisible group with generic fiber of height $g$ and special fiber of height $s$ such that the image of $\rho^{1/g}$ is open \cite{AN10}.
\end{remark}

Before proving the Theorem, let us point out the following corollary.

\begin{corollary}
\label{cor:1/gBreaks}
Suppose the image of $\rho^{1/g}$ is open in $B_{1/g}^\times$ and $N$ is as in the above theorem. Set $e_N=[K_N:K]$ and
\begin{align*}
u_N &= \inf\{x: \Gal(K_N/K)^{x}=(1)\}\\
l_N &=\inf\{x: \Gal(K_N/K)_{x}=(1)\}.
\end{align*}

Then for all $n>N$ such that
\begin{align}\label{eq:ramLie}
W(n)>l_N
\end{align}
we have
\[
\rho^{1/g}(\mathfrak{G}_{K})^{\frac{W(n)-l_N}{e_N}+u_N}=1+\pi_{1/g}^nB_{1/g}.
\]
\end{corollary}

\begin{proof}
By Proposition \ref{prop:ramLie}, if (\ref{eq:ramLie}) holds then
\[
\rho^{1/g}(\mathfrak{G}_{K})^{\frac{W(n)-l_N}{e_N}+u_N}=\rho^{1/g}(\mathfrak{G}_{K_n})^{W(n)}.
\]
The result then follows from Theorem~\ref{thm:1/gBreaks}.
\end{proof}

Theorem~\ref{thm:1/gBreaks} will follow as an immediate consequence of the following proposition and its corollary:

\begin{proposition}
\label{prop:1/gBreaks}
Suppose the image of $\rho^{1/g}$ is open in $B_{1/g}^\times$ so that there exists an $N$ such that $1+\pi_{1/g}^{n+1}B_{1/g}\subseteq \rho^{1/g}(\mathfrak{G}_K)$ for all $n\geq N$.  Then for all $n> N$, the ramification filtration of $\Gal(K_n/K_{n-1})\cong B_{1/g}/\pi_{1/g}^{}B_{1/g}\cong\F_{q^g}^+$ has a unique upper and lower break at  
\[
B(n):=-v_{n-1}\left(a_1^{-\frac{q^{g+n}}{q^g-1}}\sum_{t=1}^{q^n-1} c_t \sum_{(i_1,\dots,i_t)\in S_t} \prod_{j=1}^t a_{i_j+1}\right)/(q^g-1).
\]
\end{proposition}

\begin{proof}
Let $\varphi\in T^{1/g}(\mathcal{F})$ be invertible.
By Proposition \ref{prop:varphi property2}, $K_n=K_{n-1}(c_{q^n})$ for all $n$.  As $\rho^{1/g}(\mathfrak{G}_K)$ is open in $B_{1/g}^\times$, this implies that there exists an $N$ such that for all $n>N$ one has
\[
 \Gal(K_n/K_{n-1})\cong \frac{1+\pi_{1/g}^{n}B_{1/g}}{1+\pi_{1/g}^{n+1}B_{1/g}}\cong \frac{B_{1/g}}{\pi_{1/g}B_{1/g}}\cong\F_{q^g}^+.
\]
By (\ref{eq:FormalIDCOEF}), we have that
\[
a_1^{q^n} c_{q^n}+b=c_{q^n}^{q^g}.
\] 
Thus $z^{q^g}-a_1^{q^n}z-b$ is the minimal polynomial for $c_{q^n}$ over $K_{n-1}$. Note that the Galois conjugates of $c_{q^n}$ are
\[
c_{q^n}+i a_1^{\frac{q^n}{q^g-1}}
\]
for $i\in \F_{q^g}$. We now pass to the tame extension $K_{n-1}\left(a_1^{\frac{1}{q^g-1}}\right)$ of $K_{n-1}$. Note that passing to a tame extension merely scales the ramification break(s) by the degree of the tame extension (in our case $q^g-1$).  Consider the following generalized Artin--Schreier polynomial,
\[
\prod_{i\in \F_{q^g}} \left(x-\frac{\alpha}{a_1^{\frac{q^n}{q^g-1}}}+i\right)=z^{q^g}-z-ba_1^{-\frac{q^{g+n}}{q^g-1}},
\]
which defines the field extension $K_{n}\left(a_1^{\frac{1}{q^g-1}}\right)$ of $K_{n-1}\left(a_1^{\frac{1}{q^g-1}}\right)$. 
By (\ref{eq:FormalIDCOEF}) we have that $b=c_{q^n}^{q^g}-a_1^{q^n} c_{q^n}$. By definition we have $v_{n-1}(a_1)\geq 1$ and by Proposition \ref{prop:varphi property2} we have $v_{n}(c_{q^n})=1$. Thus $v_{n-1}(a_1^{q^n} c_{q^n})\geq q^{n-g}$ and $v_{n-1}(c_{q^n}^{q^g})=1$. Therefore, by the ultrametric triangle inequality, $v_{n-1}(b)=1$. It follows that $p$ is relatively prime to $v_{n-1}\left(ba_1^{-\frac{q^{g+n}}{q^g-1}}\right)$, since
$$
v_{n-1}\left(ba_1^{-\frac{q^{g+n}}{q^g-1}}\right)\equiv v_{n-1}(b)\equiv 1 \pmod{p} ,
$$ 
and that
\[
v_{n-1}\left(ba_1^{-\frac{q^{g+n}}{q^g-1}}\right)\leq 1-\frac{q^{g+n}}{q^g-1}<0.
\] 
By Artin--Schreier Theory (Proposition \ref{prop:Artin-Schreier_breaks}) we have that the ramification filtration of $\Gal\left(K_n\left(a_1^{\frac{1}{q^g-1}}\right)/K_{n-1}\left(a_1^{\frac{1}{q^g-1}}\right)\right)$ has a unique upper and lower break at
$-v_{n-1}\left(b a_1^{-\frac{q^{g+n}}{q^g-1}}\right)$. It follows that the ramification filtration of $\Gal(K_n/K_{n-1})$ has a unique upper and lower break at
$ -v_{n-1}\left(b a_1^{-\frac{q^{g+n}}{q^g-1}}\right)/(q^g-1)$.
\end{proof}

From the ramification breaks of $\Gal(K_n/K_{n-1})$ the upper and lower ramification filtrations of $\Gal(K_n/K_N)$ can be computed inductively using the identity of Herbrand transition functions (\ref{eq:HerbrandID}):

\begin{corollary}
 Let $B(n)$ and $W(n)$ be as in Proposition \ref{prop:1/gBreaks} and Theorem \ref{thm:1/gBreaks} respectfully. The lower ramification filtration of $\Gamma=\Gal(K_n/K_N)$ is given by
\begin{align*}
\Gamma_x&=\Gal(K_n/K_N)  && \text{ for }\ 0\leq x \leq B(N+1)\\
\Gamma_x&=\Gal(K_n/K_{N+1})  && \text{ for }\ B(N+1)< x \leq B(N+2)\\
&\vdots &&\hspace{4mm} \vdots\\
\Gamma_x&=\Gal(K_n/K_{n-1})  && \text{ for }\  B(n-1)< x\leq B(n)\\
\Gamma_x&=\Gal(K_n/K_n)=(1)  && \text{ for }\ B(N+n)<x.
\end{align*}
The upper ramification filtration of $\Gamma=\Gal(K_n/K_N)$ is given by
\begin{align*}
\Gamma^x&=\Gal(K_n/K_N)  && \text{ for }\ 0< x \leq W(N+1)\\
\Gamma^x&=\Gal(K_n/K_{N+1})  && \text{ for }\ W(N+1)< x \leq W(N+2)\\
&\vdots &&\hspace{4mm} \vdots\\
\Gamma^x&=\Gal(K_n/K_{n-1})  && \text{ for }\  W(n-1) <x\leq W(n) \\
\Gamma^x&=\Gal(K_n/K_n)=(1)  && \text{ for }\ W(n)<x.
\end{align*}
\end{corollary}

\section{Ramification Breaks of Galois Characters}
\label{sec:GalChar}

 For this section set $A=\Z_p$. We will be studying the ramification breaks of $p$-adic Lie towers associated to the Galois characters
\[
\chi_{0/1}:=\det\circ \rho^{0/1}: \mathfrak{G}_K\to \Z_p^\times.
\]
and
\[
\chi_{1/g}:=\Nm_{1/g}\circ \rho^{1/g}: \mathfrak{G}_K\to \Z_p^\times.
\]
By Theorem~\ref{thm:G2.7}, $\Nm_{1/g}\circ \rho^{1/g}$ is the inverse of $\det\circ \rho^{0/1}$ in $\Z_p$. Our strategy will be to use our result on ramification breaks of towers associated to $\rho^{1/g}$ to study the ramification breaks of $\det\circ \rho^{0/1}$. 

\begin{remark}
One may hope to determine when the image of $\chi_{0/1}(\mathfrak{G}_K)$, or equivalently $\chi_{1/g}(\mathfrak{G}_{K_N})$, is open in $\Z_p^\times$. Since the determinant and reduced norm are open maps, this is technically easier than proving that either $\rho^{1/g}$ or $\rho^{0/1}$ have open image. However, it appears that the only cases when $\chi_{0/1}$ is known to be open are in the cases when $\rho^{0/1}$ or $\rho^{1/g}$ are known to be open; namely when $d=1$, in which case $\chi_{0/1}=\rho^{0/1}$ is surjective \cite{Gro79}, and when $g=1$, in which case $\chi_{1/g}=\rho^{1/g}$ and it can be shown that the image is infinite and hence open (since all infinite closed subgroups of $\Z_p$ are open) \cite{Cha00}. 
\end{remark}

\subsection{Ramification breaks of $\chi_{0/1}$ and $\chi_{1/g}$}

The following is a corollary of Theorem~\ref{thm:1/gBreaks}:

\begin{corollary}
\label{cor:CharacterBreaks}
Suppose the image of $\rho^{1/g}$ is open in $B_{1/g}^\times$ and $N$ is as in Theorem~\ref{thm:1/gBreaks}. Then
\[
\chi_{0/1}(\mathfrak{G}_{K_N})^{W(ng)}=1+p^{n}\Z_p
\]
for all $n> N$.
\end{corollary}

\begin{proof}
By \cite[Lemma 5]{Rie70}:
\begin{align*}
\Nm_{1/g}(1+\pi_{1/g}^{n}B_{1/g})&=1+p^{\lceil n/g\rceil} \Z_p.
\end{align*}
The corollary then follows from Theorem~\ref{thm:1/gBreaks} and Theorem~\ref{thm:G2.7}.
\end{proof}

Combining the above Corollary with Corollary \ref{cor:1/gBreaks},

\begin{corollary}
Suppose the image of $\rho^{1/g}$ is open in $B_{1/g}^\times$ and $N$ is as in Theorem~\ref{thm:1/gBreaks}. Then for all sufficiently large $n$,
\[
\chi_{0/1}(\mathfrak{G}_{K})^{aW(ng)+b}=1+p^n\Z_p
\]
for some fixed constants $a,b$.
\end{corollary}

\appendix
\section{A mistaken generalization of Tate's Lemma}\label{appendix}

The proof of \cite[Theorem 4.6]{Cha00}, which gives a comparison between the Lie and Ramification filtrations in the case that $g=1$, relies upon \cite[Lemma 4.4]{Cha00}, a generalized version of a lemma of Tate. Unfortunately, as pointed out by a referee, this generalization is false, as we now explain.

The original version of Tate's Lemma \cite[Lemma 1.5]{Gro79} relates upper ramification breaks of $\Gal(K(\alpha)/K)$, where $\alpha$ is a root of a degree $n$ Eisenstein polynomial $f(x)$, to the the Newton polygon of the ramification polynomial 
\[
\alpha^{-n} f(\alpha x+\alpha).
\]
Chai's proposed generalization of Tate's Lemma \cite[Lemma 4.4]{Cha00} claims that if $f(x)\in K[x]$ is a polynomial with the property that for each root $\alpha$ of $f(x)$ the extension $K(\alpha)$ is totally ramified, and if $\pi$ is a uniformizer of $K(\alpha)$, then the upper ramification breaks of $\Gal(K(\alpha)/K)$ can be determined from the Newton polygon of the polynomial
\[
\pi^{-n} f(\pi x+\alpha).
\]
A referee provided the following counterexample to this generalization of Tate's Lemma.

\begin{example}
Let $K=\F_q((t))$ and $f(x)=x^p-x+t^{-r}$ with $p\nmid r$. Adjoining any root $\alpha$ of $f(x)$ to $K$ gives a degree $p$ totally ramified extension $K(\alpha)$ of $K$. By Artin--Schreier Theory (see \cite{Has35} or \cite[IV \S 2 Exercise 5]{Ser67}), $\Gal(K(\alpha)/K)$ has a unique upper ramification break at $r$. Let $\pi$ be a uniformizer of $K(\alpha)$. As the Galois conjugates of $\alpha$ are $\alpha+1$, $\alpha+2$, $\dots$, $\alpha+p-1$, we have that 
\[
\pi^{-p} f(\pi x+\alpha)=x\left(x-\frac{1}{\pi}\right)\left(x-\frac{2}{\pi}\right)\cdots \left(x-\frac{p-1}{\pi}\right).
\]
But the Newton polygon of this polynomial does not depend on $r$, and therefore it can not determine the upper ramification break in these cases.
\end{example}

\end{document}